\documentclass[11pt,a4paper]{article}

\usepackage{amsmath}
\usepackage{amsthm}
\usepackage{amssymb}
\usepackage{cain_arxiv}

\theoremstyle{definition}
\newtheorem{definition}{Definition}[section]
\newtheorem{example}[definition]{Example}

\theoremstyle{plain}
\newtheorem{proposition}[definition]{Proposition}
\newtheorem{lemma}[definition]{Lemma}
\newtheorem{theorem}[definition]{Theorem}

\numberwithin{equation}{section}

\def\fullref#1#2{%
  \ifdefined\hyperref%
    {\hyperref[#2]{#1\space\penalty 200\relax\ref*{#2}}}%
  \else%
    {#1\space\penalty 200\relax\ref{#2}}%
  \fi%
}

\newcommand{\defterm}[1]{\textit{#1}}

\newcommand{\fsa}[1]{\mathfrak{#1}}

\newcommand{\pres}[2]{\left\langle #1\:|\:#2 \right\rangle}

\newcommand{\nset}{\mathbb{N}}

\newcommand{\emptyword}{\varepsilon}
\newcommand{\rel}[1]{\mathcal{#1}}

\newcommand{\imderives}{\Rightarrow}
\newcommand{\derives}{\Rightarrow^*}

\newcommand{\rev}{\mathrm{rev}}

\newcommand{\dash}{\unskip\space---\space}

\begin{document}




\title{Context-free rewriting systems and word-hy\-per\-bol\-ic structures with uniqueness}
\author{Alan J. Cain \& Victor Maltcev}
\date{}
\thanks{The first author's research was funded by the European
  Regional Development Fund through the programme {\sc COMPETE} and by
  the Portuguese Government through the {\sc FCT} (Funda\c{c}\~{a}o
  para a Ci\^{e}ncia e a Tecnologia) under the project
  {\sc PEst-C}/{\sc MAT}/{\sc UI}0144/2011 and through an {\sc FCT} Ci\^{e}ncia 2008
  fellowship.}

\maketitle

\address[AJC]{%
Centro de Matem\'{a}tica, Universidade do Porto, \\
Rua do Campo Alegre 687, 4169--007 Porto, Portugal
}
\email{%
ajcain@fc.up.pt
}
\webpage{%
www.fc.up.pt/pessoas/ajcain/
}
\address[VM]{%
School of Mathematics \& Statistics, University of St Andrews,\\
North Haugh, St Andrews, Fife KY16 9SS, United Kingdom
}
\email{%
victor@mcs.st-andrews.ac.uk
}

\begin{abstract}
This paper proves that any monoid presented by a confluent
con\-text-free monadic rewriting system is word-hy\-per\-bol\-ic. This result then
applied to answer a question asked by Duncan \& Gilman by exhibiting
an example of a word-hy\-per\-bol\-ic monoid that does not admit a
word-hy\-per\-bol\-ic structure with uniqueness (that is, in which the language
of representatives maps bijectively onto the monoid).
\end{abstract}


\section{Introduction}

Hyperbolic groups \dash groups whose Cayley graphs are
hy\-per\-bol\-ic metric spaces \dash have grown into one of the most
fruitful areas of group theory since the publication of Gromov's
seminal paper \cite{gromov_hyperbolic}. The concept of hyperbolicity
generalizes to semigroups and monoids in more than one way. First, one
can consider semigroups and monoids whose Cayley graphs are
hy\-per\-bol\-ic \cite{c_rshcg,cassaigne_infinitewords}. Second, one
can use Gilman's characterization of hy\-per\-bol\-ic groups using
con\-text-free languages \cite{gilman_hyperbolic}. This
characterization says that a group $G$ is hy\-per\-bol\-ic if and only if
there is a regular language $L$ (over some generating set) of normal
forms for $G$ such that the language
\[
M(L) = \big\{u\#_1v\#_2w^\rev : u,v,w \in L \land uv =_G w\big\}
\]
(where $w^\rev$ denotes the reverse of $w$) is con\-text-free. (The pair
$(L,M(L))$ is called a \defterm{word-hy\-per\-bol\-ic structure}.) Duncan \&
Gilman \cite{duncan_hyperbolic} pointed out that this characterization
generalizes naturally to semigroups and monoids. The geometric
generalization gives rise to the notion of \defterm{hy\-per\-bol\-ic}
semigroup; the linguistic one to the notion of
\defterm{word-hy\-per\-bol\-ic} semigroups. While the two notions are
equivalent for groups \cite[Corollary~4.3]{duncan_hyperbolic} and more
generally for completely simple semigroups
\cite[Theorem~4.1]{fountain_hyperbolic}, they are not equivalent for
general semigroups. This paper is concerned with
\emph{word-hy\-per\-bol\-ic} semigroups.

Some of the pleasant properties of hy\-per\-bol\-ic groups do not generalize
to word-hy\-per\-bol\-ic semigroups. For example, hy\-per\-bol\-ic groups are
always automatic \cite[Theorem~3.4.5]{epstein_wordproc};
word-hy\-per\-bol\-ic semigroups may not even be \emph{asynchronously}
automatic \cite[Example~7.7]{hoffmann_relatives}. On the other hand,
word-hyperbolicity for semigroups is independent of the choice of
generating set \cite[Theorem~3.4]{duncan_hyperbolic}, unlike
automaticity for semigroups \cite[Example~4.5]{campbell_autsg}.

Duncan \& Gilman \cite[Question~2]{duncan_hyperbolic} asked whether
every word-hy\-per\-bol\-ic monoid admits a word-hy\-per\-bol\-ic structure where
the language of representatives $L$ projects \emph{bijectively} onto
the monoid. By analogy with the case of automatic
groups~\cite[\S~2.5]{epstein_wordproc} and
semigroups~\cite[p.~380]{campbell_autsg}, such a word-hy\-per\-bol\-ic
structure is called a \defterm{word-hy\-per\-bol\-ic structure with
  uniqueness}. The question also applies to semigroups; Duncan \&
Gilman have a particular interest in the situation for monoids because a
positive answer in that case would imply that the class of
word-hy\-per\-bol\-ic semigroups is closed under adjoining an identity.

As explained in \fullref{Subsection}{subsec:wordhyp} below, every
hyperbolic group admits a word-hyperbolic structure with
uniqueness. Furthermore, every automatic semigroup admits an automatic
structure with uniqueness \cite[Corollary~5.6]{campbell_autsg}.

The main goal of this paper is to give a negative answer to the
question of Duncan \& Gilman by exhibiting an example of a
word-hy\-per\-bol\-ic monoid that does not admit a
word-hy\-per\-bol\-ic structure with uniqueness
(\fullref{Example}{ex:wordhypnotunique}). En route, however, a result
of independent interest is proven: that any monoid presented by a
confluent con\-text-free monadic rewriting system is
word-hy\-per\-bol\-ic (\fullref{Theorem}{thm:contextfreesrs}).

\section{Preliminaries}
\label{sec:preliminaries}

This paper assumes familiarity with regular languages and finite
automata and with con\-text-free grammars and languages; see
\cite[Chs~2--4]{hopcroft_automata} for background reading and for the
notation used here.

The empty word (over any alphabet) is denoted $\emptyword$.

\subsection{Word-hyperbolicity}
\label{subsec:wordhyp}

\begin{definition}
\label{def:wordhyp}
A \defterm{word-hy\-per\-bol\-ic structure} for a semigroup $S$ is a
pair $(L,M(L))$, where $L$ is a regular language over an alphabet $A$
representing a finite generating set for $S$ such that $L$ maps onto
$S$, and where
\[
M(L) = \{u\#_1v\#_2w^\rev : u,v,w \in L \land uv =_S w\}
\]
(where $\#_1$ and $\#_2$ are new symbols not in $A$ and $w^\rev$
denotes the reverse of the word $w$) is con\-text-free. The pair
$(L,M(L)$ is a \defterm{word-hy\-per\-bol\-ic structure with uniqueness} if
$L$ maps bijectively onto $S$; that is, if every element of $S$ has a
unique representative in $L$.

A semigroup is \defterm{word-hy\-per\-bol\-ic} if it admits a
word-hy\-per\-bol\-ic structure.
\end{definition}

A group is hy\-per\-bol\-ic in the sense of Gromov~\cite{gromov_hyperbolic}
if and only if it admits a word-hy\-per\-bol\-ic structure
(\cite[Theorem~1]{gilman_hyperbolic} and
\cite[Corollary~4.3]{duncan_hyperbolic}). Furthermore, every group
admits a word-hy\-per\-bol\-ic structure with uniqueness: if $(L,M(L))$ is a
word-hy\-per\-bol\-ic structure for a group $G$, then the fellow-traveller
property is satisfied \cite[Theorem~4.2]{duncan_hyperbolic} and so $L$
forms part of an automatic structure for $G$
\cite[Theorem~2.3.5]{epstein_wordproc}. Therefore there exists an
automatic structure with uniqueness for $G$, where the language of
representatives $L'$ is a subset of $L$
\cite[Theorem~2.5.1]{epstein_wordproc}. Hence $(L',M(L) \cap
L'\#_1L'\#_2(L')^\rev)$ is a word-hy\-per\-bol\-ic structure with uniqueness
for $G$.

\subsection{Rewriting systems}

This subsection contains facts about string rewriting needed later in
the paper. For further background information, see
\cite{book_srs}.

A \defterm{string rewriting system}, or simply a \defterm{rewriting
  system}, is a pair $(A,\rel{R})$, where $A$ is a finite alphabet and
$\rel{R}$ is a set of pairs $(\ell,r)$, known as \defterm{rewriting
  rules}, drawn from $A^* \times A^*$. The single reduction relation
$\imderives_{\rel{R}}$ is defined as follows: $u \imderives_{\rel{R}} v$ (where $u,v \in
A^*$) if there exists a rewriting rule $(\ell,r) \in \rel{R}$ and
words $x,y \in A^*$ such that $u = x\ell y$ and $v = xry$. That is, $u
\imderives_{\rel{R}} v$ if one can obtain $v$ from $u$ by substituting the word
$r$ for a subword $\ell$ of $u$, where $(\ell,r)$ is a rewriting
rule. The reduction relation $\derives_{\rel{R}}$ is the reflexive and
transitive closure of $\imderives_{\rel{R}}$. The process of replacing a subword
$\ell$ by a word $r$, where $(\ell,r) \in \rel{R}$, is called
\defterm{reduction}, as is the iteration of this process.

A word $w \in A^*$ is \defterm{reducible} if it contains a subword $\ell$
that forms the left-hand side of a rewriting rule in $\rel{R}$; it is
otherwise called \defterm{irreducible}.

The string rewriting system $(A,\rel{R})$ is \defterm{noetherian} if
there is no infinite sequence $u_1,u_2,\ldots \in A^*$ such that $u_i
\imderives_{\rel{R}} u_{i+1}$ for all $i \in \nset$. That is, $(A,\rel{R})$ is
noetherian if any process of reduction must eventually terminate with
an irreducible word. The rewriting system $(A,\rel{R})$ is
\defterm{confluent} if, for any words $u, u',u'' \in A^*$ with $u
\derives_{\rel{R}} u'$ and $u \derives_{\rel{R}} u''$, there exists a word $v \in A^*$
such that $u' \derives_{\rel{R}} v$ and $u'' \derives_{\rel{R}} v$. 

The string rewriting system $(A,\rel{R})$ is \defterm{length-reducing}
if $(\ell,r) \in \rel{R}$ implies that $|\ell| > |r|$. Observe that
any length-reducing rewriting system is necessarily noetherian.
The rewriting system $(A,\rel{R})$ is
 \defterm{monadic} if it is
length-reducing and the right-hand side of each rule in $\rel{R}$ lies
in $A \cup \{\emptyword\}$; it is 
\defterm{special} if it is
length-reducing and each right-hand side is the empty word
$\emptyword$. Observe that every special rewriting system is also monadic.

A special or monadic rewriting system $(A,\rel{R})$ is con\-text-free
if, for each $a \in A \cup\{\emptyword\}$, the set of all left-hand
sides of rules in $\rel{R}$ with right-hand side $a$ is a context-free
language.

Let $(A,\rel{R})$ be a confluent noetherian string rewriting
system. Then for any word $u \in A^*$, there is a unique irreducible
word $v \in A^*$ with $u \derives_{\rel{R}} v$
\cite[Theorem~1.1.12]{book_srs}. The irreducible words are said to be
in \defterm{normal form}. The monoid presented by $\pres{A}{\rel{R}}$
may be identified with the set of normal form words under the
operation of `concatenation plus reduction to normal form'.



The subscript symbols in the derivation and one-step derivation
relations $\derives_{\rel{R}}$ and $\imderives_{\rel{R}}$ for a
rewriting system $\rel{R}$ are never omitted in this paper, in order
to avoid any possible confusion with the derivation and one-step
derivation relations $\derives_\Gamma$ and $\imderives_\Gamma$ for a
context-free grammar $\Gamma$.

\section{Monoids presented by confluent con\-text-free monadic rewriting systems}

\begin{theorem}
\label{thm:contextfreesrs}
Let $(A,\rel{R})$ be a confluent con\-text-free monadic rewriting
system. Then $(A^*,M(A^*))$ is a word-hy\-per\-bol\-ic structure for the monoid
presented by $\pres{A}{\rel{R}}$.
\end{theorem}

\begin{proof}
Let $M$ be the monoid presented by $\pres{A}{\rel{R}}$. Let
\[
K = \{u\#_2v^\rev : u,v \in A^*, u =_M v\}.
\]
Let $\phi : (A \cup \{\#_1,\#_2\})^* \to (A \cup \{\#_1\})^*$ be the homomorphism extending
\[
\#_1 \mapsto \emptyword,\quad \#_2 \mapsto \#_2,\quad a \mapsto a \text{ for all $a \in A$}.
\]
Then $M(A^*) = K\phi^{-1} \cap A^*\#_1A^*\#_2A^*$. Since the class of context-free languages
is closed under taking inverse homomorphisms, to prove that $M(A^*)$ is
context-free it suffices to prove that $K$ is context-free.

For each $a
\in A \cup \{\emptyword\}$, let $\$_a$ and $\tilde{\$}_a$ be new
symbols. Let
\begin{align*}
\$_{A \cup \{\emptyword\}} &= \{\$_a : a \in A \cup \{\emptyword\}\} & \$_A &= \{\$_a : a \in A\}, \\
\tilde{\$}_{A \cup \{\emptyword\}} &= \{\tilde{\$}_a : a \in A \cup \{\emptyword\}\} & \tilde{\$}_A &= \{\tilde{\$}_a : a \in A\},
\end{align*}
and for any word $w = w_1\cdots w_n$ with $w_i \in A$, let $\$_w$ and
$\tilde{\$}_w$ be abbreviations for $\$_{w_1}\cdots\$_{w_n}$ and
$\tilde{\$}_{w_1}\cdots\tilde{\$}_{w_n}$ respectively.

For each $a \in A \cup \{\emptyword\}$, let $\Gamma_a =
(N_a,A,P_a,O_a)$ be a con\-text-free grammar such that $L(\Gamma_a)$ is
the set of left-hand sides of rewriting rules in $\rel{R}$ whose
right-hand side is $a$. Since $\rel{R}$ is length-reducing, no
$L(\Gamma_a)$ contains $\emptyword$. Therefore assume without loss of
generality that no $\Gamma_a$ contains a production whose right-hand
side is $\emptyword$ \cite[Theorem~4.3]{hopcroft_automata}.

Modify each $\Gamma_a$ by replacing each appearance of a terminal
letter $b \in A$ in a production by $\$_b$; the grammar $\Gamma'_a =
(N'_a,\$_{A \cup \{\emptyword\}},P'_a,O'_a)$ thus formed has the
property that $w \in L(\Gamma_a)$ if and only if $\$_w \in
L(\Gamma'_a)$. Modify each $\Gamma_a$ by reversing the right-hand side
of every production in $P_a$ and by replacing each appearance of a
terminal letter $b \in A$ in a production by $\tilde{\$}_b$; the grammar
$\Gamma''_a = (N''_a,\$_{A \cup \emptyword},P''_a,O''_a)$ thus
produced has the property that $w \in L(\Gamma_a)$ if and only if
$\tilde{\$}_{w^\rev} \in L(\Gamma''_a)$.

The language
\[
\{\$_p\#_2\tilde{\$}_{p^\rev} : p \in A^*\}
\]
is clearly con\-text-free. (Notice that $\$_p$ can either be an
abbreviation for a non-empty word $\$_{p_1}\cdots\$_{p_k}$ or the
single letter $\$_\emptyword$, and similarly for $\tilde{\$}_{p^\rev}$.) Let $\Delta = (N_\Delta,\$_A \cup \tilde{\$}_A
\{\#_2\},P_\Delta,O_\Delta)$ be a con\-text-free grammar defining
this language. Assume without loss of generality that the various non-terminal
alphabets $N'_a$, $N''_a$ and $N_\Delta$ are pairwise disjoint.

Define a new con\-text-free grammar $\Theta =
(N_\Theta,A\cup\{\#_2\},P_\Theta,O_\Delta)$ by letting
\[
N_\Theta = N_\Delta \cup \$_{A \cup \{\emptyword\}} \cup \tilde{\$}_{A \cup \{\emptyword\}} \cup \bigcup_{a \in A \cup\{\emptyword\}} (N'_a \cup N''_a),
\]
and
\begin{align}
P_\Theta &= P_\Delta \cup \Bigl[\bigcup_{a \in A\cup\{\emptyword\}} (P'_a \cup P''_a)\Bigr] \nonumber\\
&\qquad\cup \bigl\{\$_a \to \$_a\$_\emptyword, \$_a \to \$_\emptyword\$_a, \tilde{\$}_a \to \tilde{\$}_a\tilde{\$}_\emptyword, \tilde{\$}_a \to \tilde{\$}_\emptyword\tilde{\$}_a : a \in A \cup \{\emptyword\}\bigr\} \label{eq:insprods}\\
&\qquad\cup \bigl\{\$_a \to O'_a, \tilde{\$}_a \to O''_a : a \in A \cup \{\emptyword\}\bigr\} \label{eq:startprods}\\
&\qquad\cup \bigl\{\$_a \to a, \tilde{\$}_a \to a : a \in A \cup \{\emptyword\}\bigr\}.\label{eq:endprods}
\end{align}
Notice that elements of $\$_{A\cup\{\emptyword\}}$ now play the
r\^{o}le of non-terminals, while in the various grammars $\Gamma'_a$
and $\Gamma''_a$, they were terminals. Notice further that the start
symbol of $\Theta$ is $O_\Delta$.

The aim is now to show that $L(\Theta) = K$.

\begin{lemma}
\label{lem:lthetasubset}
If $w \in L(\Theta)$, then $w = u\#_2v^\rev$ for some $u,v \in A^*$,
and there exists some $p \in A^*$ such that $\$_p \derives_\Theta u$
and $\tilde{\$}_{p^\rev} \derives_\Theta v^\rev$.
\end{lemma}

\begin{proof}
Let $w \in L(\Theta)$. Then $O_\Delta \derives_\Theta w$, and the first production
applied is from $P_\Delta$. Since no production in $P_\Theta -
P_\Delta$ introduces a non-terminal symbol from $N_\Delta$, assume
that all productions from $P_\Delta$ in the derivation of $w$ are
carried out first, before any productions from $P_\Theta -
P_\Delta$. This shows that there is some word $q \in L(\Delta)$ such
that $O_\Delta \derives_\Theta q \derives_\Theta w$. By the definition
of $\Delta$, it follows that $q = \$_p\#_2\tilde{\$}_{p^\rev}$ with
\[
O_\Delta \derives_\Theta \$_p\#_2\tilde{\$}_{p^\rev} \derives_\Theta w.
\]
Since symbols from $\$_{A \cup \{\emptyword\}} \cup \tilde\$_{A \cup
  \{\emptyword\}}$ can ultimately only derive symbols from $A$ (and
not the symbol $\#_2$), it follows that there exist $u,v \in A^*$ with
$\$_p \derives_\Theta u$ and $\tilde{\$}_{p^\rev} \derives_\Theta v^\rev$ such
that $w = u\#_2v^\rev$.
\end{proof}

\begin{lemma}
\label{lem:rsderivimpliescfgderiv}
Let $w,u \in A^*$. If $w \derives_{\rel{R}} u$, then $\$_u \derives_\Theta \$_w$.
\end{lemma}

\begin{proof}
Suppose
\[
w = w_0 \imderives_{\rel{R}} w_1 \imderives_{\rel{R}} w_2 \imderives_{\rel{R}} \ldots \imderives_{\rel{R}} w_n = u
\]
is a sequence of rewriting of minimal length from $w$ to $u$.

Proceed by induction on $n$. If $n = 0$, it follows that $w = u$ and
there is nothing to prove. So suppose $n > 0$ and that the result
holds for all shorter such minimal-length rewriting sequences. Then
$w_0 \imderives_{\rel{R}} w_1$, and so $w_0 = x\ell y$ and $w_1 = xay$
for some $x,y \in A^*$, $a \in A \cup \{\emptyword\}$, and $(\ell,a)
\in \rel{R}$. So $\ell \in L(\Gamma_a)$. Hence, first applying a
production of type \eqref{eq:startprods}, the construction of
$\Gamma'_a$ and the inclusion of all its productions in $\Theta$ shows
that
\begin{equation}
\label{eq:derivinglhs}
\$_a \imderives_\Theta O'_a \derives_\Theta \$_\ell.
\end{equation}
By the induction hypotheses, $\$_u
\derives_\Theta \$_{w_1}$. Now consider the cases $a \in A$ and $a =
\emptyword$ separately:
\begin{enumerate}

\item $a \in A$. Then $\$_{w_1} = \$_{x}\$_{a}\$_{y}$ and so 
\begin{align*}
\$_u &\derives_\Theta \$_{w_1}  && \text{(by the induction hypothesis)}\\
&\;\rlap{$=$}\phantom{\derives_\Theta}\; \$_x\$_a\$_y \\
&\derives_\Theta \$_x\$_\ell\$_y  && \text{(by \eqref{eq:derivinglhs})}\\
&\;\rlap{$=$}\phantom{\derives_\Theta}\; \$_{w_0} \\
&\;\rlap{$=$}\phantom{\derives_\Theta}\; \$_w.
\end{align*}

\item $a = \emptyword$. Then $\$_{w_1} = \$_{x}\$_{y}$ and so by \eqref{eq:derivinglhs},
\begin{align*}
\$_u &\derives_\Theta \$_{w_1} && \text{(by the induction hypothesis)}\\
 &\derives_\Theta \$_x\$_y  \\
&\imderives_\Theta \$_x\$_a\$_y  && \text{(by \eqref{eq:insprods})}\\
&\derives_\Theta \$_x\$_\ell\$_y  && \text{(by \eqref{eq:derivinglhs})}\\
&\;\rlap{$=$}\phantom{\derives_\Theta}\; \$_{w_0} \\
&\;\rlap{$=$}\phantom{\derives_\Theta}\; \$_w.
\end{align*}
\end{enumerate}
This completes the proof.
\end{proof}

\begin{lemma}
\label{lem:cfgderivimpliesrsderiv}
Let $u,w \in A^*$. If $\$_u \derives_\Theta \$_w$, then $w \derives_{\rel{R}} u$.
\end{lemma}

\begin{proof}
The strategy is to proceed by induction on the number $n$ of productions of type
\eqref{eq:insprods} or \eqref{eq:startprods} in the minimal-length
derivation of $\$_w$ from $\$_u$.

Suppose such a minimal length derivation involves a production
$\$_\emptyword \to \emptyword$ (of type \eqref{eq:endprods}). If this
symbol $\$_\emptyword$ is introduced by a production of type
\eqref{eq:insprods}, then the derivation would not be of minimal
length. So this symbol $\$_\emptyword$ must be present in $\$_u$,
which, by the definition of the abbreviation $\$_u$ requires $u =
\emptyword$. But this would mean that the derivation produced
$\emptyword$, which contradicts the hypothesis of the lemma. So the
derivation does not involve productions $\$_\emptyword \to
\emptyword$.

The only productions where symbols from $\$_{A \cup \{\emptyword\}}$
appear on the left-hand side are of types \eqref{eq:insprods},
\eqref{eq:startprods}, and \eqref{eq:endprods}. Since there are no
productions $\$_\emptyword \to \emptyword$, any production of type
\eqref{eq:endprods} would produce a terminal symbol, which is
impossible. So the first production applied in the derivation sequence
must be of type \eqref{eq:insprods} or \eqref{eq:startprods}.

Suppose first that $n = 0$. Then there is no possible first production
and thus $\$_w = \$_u$, which entails $w = u$ and so there is nothing to
prove.

Suppose now that $n > 0$ and that the result holds for all shorter
such minimal-length derivations. Consider cases separately depending
on the whether the first production applied in the derivation is of type
\eqref{eq:insprods} or \eqref{eq:startprods}.:

\begin{enumerate}

\item Type \eqref{eq:insprods}. So $\$_u = \$_x\$_y \imderives_\Theta
  \$_x\$_\emptyword\$_y$ for some $x,y \in A^*$ with $xy = u$. The
  symbol $\$_\emptyword$ thus produced does not derive $\emptyword$
  since no production $\$_\emptyword \to \emptyword$ is involved. So
  $\$_x \derives_\Theta \$_{w'}$, $\$_\emptyword \derives_\Theta
  \$_{w''}$, $\$_y \derives_\Theta \$_{w'''}$, where $w = w'w''w'''$,
  where $w',w''' \in A^*$ and $w'' \in A^+$ and all three of these
  derivations involve fewer than $n$ productions of type
  \eqref{eq:insprods} or \eqref{eq:startprods}. By the induction
  hypothesis, $w' \derives_{\rel{R}} x$, $w'' \derives_{\rel{R}}
  \emptyword$, and $w''' \derives_{\rel{R}} y$, and thus $w =
  w'w''w''' \derives_{\rel{R}} xy = u$.

\item Type \eqref{eq:startprods}. So $\$_u = \$_x\$_a\$_y
  \imderives_\Theta \$_xO'_a\$_y$ for some $x,y \in A^*$ with $xay =
  u$. Now, $O'_a$ is the start symbol of $\Gamma'_a$, and $L(\Gamma'_a)$
  consists of words of the form $\$_\ell$ where $\ell
  \imderives_{\rel{R}} a$. Thus
\[
\$_u \imderives_\Theta \$_xO'_a\$_y \derives_\Theta \$_x\$_\ell\$_y \derives_\Theta \$_w.
\]
Thus $\$_x \derives_\Theta \$_{w'}$, $\$_\ell \derives_\Theta
\$_{w''}$, and $\$_y \derives_\Theta \$_{w'''}$, where $w',w'',w'''
\in A^*$ are such that $w = w'w''w'''$, and each of these
derivation sequences involve fewer than $n$ productions of type
\eqref{eq:insprods} or \eqref{eq:startprods}. Hence by the induction
hypothesis, $w' \derives_{\rel{R}} x$, $w'' \derives_{\rel{R}} \ell$,
and $w''' \derives_{\rel{R}} y$. Therefore
\[
w = w'w''w''' \derives_{\rel{R}} x\ell y \imderives_{\rel{R}} xay = u.
\]
\end{enumerate}
This completes the proof.
\end{proof}

\begin{lemma}
\label{lem:rsderiviffcfgderiv}
For any $u,w \in A^*$
$w \derives_{\rel{R}} u$ if and only if $\$_u \derives_\Theta w$.
\end{lemma}

\begin{proof}
Suppose $w \derives_{\rel{R}} u$. Then $\$_u \derives_\Theta \$_w$ by
\fullref{Lemma}{lem:rsderivimpliescfgderiv}. By $|w|$ applications of
productions of type \eqref{eq:endprods}, $\$_w \derives_\Theta w$. So
$\$_u \derives_\Theta w$.

Suppose that $\$_u \derives_\Theta w$. Only
productions of type \eqref{eq:endprods} have terminals on the
right-hand side. So $\$_u \derives_\Theta \$_w \derives_\Theta
w$. So by \fullref{Lemma}{lem:cfgderivimpliesrsderiv}, $w
\derives_{\rel{R}} u$.
\end{proof}

Reasoning symmetric to the proofs of
\fullref{Lemmata}{lem:rsderivimpliescfgderiv},
\ref{lem:cfgderivimpliesrsderiv}, and \ref{lem:rsderiviffcfgderiv}
establishes the following result:

\begin{lemma}
\label{lem:rsderiviffcfgderivrev}
For any $u,w \in A^*$
$w \derives_{\rel{R}} u$ if and only if $\tilde{\$}_{u^\rev} \derives_\Theta w^\rev$.
\end{lemma}

Suppose $u\#_2v^\rev \in K$. Then $u,v \in A^*$ and $u =_M
v$. Therefore there is a normal form word $p$ with $u
\derives_{\rel{R}} p$ and $v \derives_{\rel{R}} p$. So by
\fullref{Lemmata}{lem:rsderiviffcfgderiv} and
\ref{lem:rsderiviffcfgderivrev}, $\$_p \derives_\Theta u$ and
$\tilde{\$}_{p^\rev} \derives_\Theta v^\rev$. Since every production
in $P_\Delta$ is included in $P_\Theta$, it follows that
\[
O_\Delta \derives_\Theta \$_p\#_2\tilde{\$}_{p^\rev},
\]
whence $O_\Delta \derives_\Theta u\#_2 v^\rev$ and so
$u\#_2v^\rev \in L(\Theta)$.

Conversely, suppose $w \in L(\Theta)$. By
\fullref{Lemma}{lem:lthetasubset}, there are words $u,v,p \in A^*$
with $w = u\#_2v^\rev$, $\$_p \derives_\Theta u$, and $\tilde\$_{p^\rev}
\derives_\Theta v^\rev$. By \fullref{Lemmata}{lem:rsderiviffcfgderiv} and
\ref{lem:rsderiviffcfgderivrev}, it follows that $u \derives_{\rel{R}}
p$ and $v \derives_{\rel{R}} p$. So $u =_M v$ and thus $w = u\#_2v^\rev
\in K$.

Hence $L(\Theta) = K$. Thus $K$ and so $M(A^*)$ are
context-free. Therefore $(A^*,M(A^*))$ is a word-hy\-per\-bol\-ic
structure for the monoid $M$.
\end{proof}

\section{Word-hy\-per\-bol\-ic structures with uniqueness}

This section exhibits an example of a word-hy\-per\-bol\-ic monoid that
does not admit a word-hy\-per\-bol\-ic structure with uniqueness.

The following preliminary result, showing that admitting a
word-hy\-per\-bol\-ic structure with uniqueness is not dependent on the
choice of generating set, is needed. The proof is similar to that of
the independence of word-hyperbolicity from the choice of generating
set \cite[Theorem~3.4]{duncan_hyperbolic}, but the detail and
exposition are different to make clear that uniqueness is
preserved. Additionally, the result here also shows that whether one
deals with monoid or semigroup generating sets is not a concern.

\begin{proposition}
\label{prop:monoiduniquenesschangegen}
Let $M$ be a monoid that admits a word-hy\-per\-bol\-ic structure with
uniqueness over either a semigroup or monoid generating set, and let
$A$ be a finite alphabet representing a semigroup or monoid generating
set for $M$. Then there is a language $L$ such that $(A,L)$ is a
word-hy\-per\-bol\-ic structure with uniqueness for $M$.
\end{proposition}

\begin{proof}
Suppose $S$ admits a word-hy\-per\-bol\-ic structure $(B,K)$. For each $b
\in B$, let $u_b \in A^*$ be such that $u_b =_M b$. (If
$A$ represents a semigroup generating set, ensure that $u_b$ lies in
$A^+$; this restriction is important only if $b$ is actually the
identity.)  Let $\rel{P} \subseteq B^* \times A^*$ be the rational relation:
\[
\rel{P} = \bigl(\{(b,u_b) : b \in B\}\bigr)^*
\]
Notice that if $(v,w) \in \rel{P}$, then $v =_M w$.

Let
\[
L = K \circ \rel{P} = \bigl\{w \in A^* : (\exists v \in K)((v,w) \in \rel{P})\bigr\};
\]
observe that $L$ is a regular language. Notice that, by the definition
of $\rel{P}$, for each word $v$ in $K$ there is exactly one word $w \in L$
with $(v,w) \in \rel{P}$. Since for each $x \in M$ there is exactly one word
$v$ in $K$ with $v =_M x$, it follows that there is exactly one
word $w \in L$ with $w =_M x$. That is, the language $L$ maps
bijectively onto $M$.

Let $\rel{Q}$ be the rational relation
\[
\rel{P}(\#_1,\#_1)\rel{P}(\#_2,\#_2)\rel{P}^\rev.
\]
Then $M(L) = M(K) \circ \rel{Q}$ and so $M(L)$ is a
con\-text-free language.

Thus $(A,L)$ is a word-hy\-per\-bol\-ic structure for $S$ in every
case except when $S$ is a monoid, $A$ is a semigroup generating set,
and the representative in $K$ of the identity is $\emptyword$. In this
case, let $L_1 = (L - \{\emptyword\}) \cup \{e\}$, where $e \in A^+$
representings the identity. Then $L_1$ is contained in $A^+$ and maps
bijectively onto $S$. The language $M(L_1)$ is con\-text-free: a
pushdown automaton recognizing it can be constructed from one
recognizing $M(L)$ by modifying it to read $e$ instead of the empty
word as one of the multiplicands or result while simulating reading
the empty word whenever $e$ is encountered.
\end{proof}




\begin{example}
\label{ex:wordhypnotunique}
Let $A = \{a,b,c,d\}$ and let $\rel{R} = \{(ab^\alpha c^\alpha
d,\emptyword) : \alpha \in \nset\}$. Let $M$ be the monoid presented
by $\pres{A}{\rel{R}}$. Then $M$ is word-hy\-per\-bol\-ic but does not admit
a regular language of unique representatives and thus, in particular,
does not admit a word-hy\-per\-bol\-ic structure with uniqueness.
\end{example}

\begin{proof}
Let $G$ be the language of left-hand sides of rewriting rules in
$\rel{R}$. The language $G$ is con\-text-free, and so
$(A,\rel{R})$ is a con\-text-free special rewriting system. Two
left-hand sides of rewriting rules in $\rel{R}$ only overlap if they
are exactly equal, and so $(A,\rel{R})$ is confluent. Hence, by
\fullref{Theorem}{thm:contextfreesrs}, $(A^*,M(A^*))$ is a word-hy\-per\-bol\-ic
structure for the monoid $M$. So $M$ is word-hy\-per\-bol\-ic. Identify $M$
with the language of normal form words of $(A,\rel{R})$.

Suppose for \textit{reductio ad absurdum} that $M$ admits a
word-hy\-per\-bol\-ic structure with uniqueness. Then, by
\fullref{Proposition}{prop:monoiduniquenesschangegen}, there is a regular
language $L$ over $A$ such that $(L,M(L))$ is a word-hy\-per\-bol\-ic
structure with uniqueness for $M$. In particular, every element of $M$
has a unique representative in $L$. Let $\fsa{A}$ be a finite state
automaton recognizing $L$ and let $n$ be the number of states in
$\fsa{A}$.

Now, if $w \in L$ represents $u \in M$, then $w \derives_{\rel{R}} u$: the word
$u$ can be obtained from $w$ by replacing subwords lying in $G$ by the
empty word, which effectively means deleting subwords that lie in
$G$. Consider this process in reverse: the word $w$ can be obtained
from $u$ by inserting words from $G$. 

If a word from $G$ is inserted between two letters of $u$, call it a
\defterm{depth-$1$} inserted word. If a word from $G$ is inserted
between two letters of a depth-$k$ inserted word, it is called a
depth-$(k+1)$ inserted word. A word inserted immediately before the
first letter or immediately after the last letter of a depth-$k$
inserted word also counts as a depth-$k$ inserted word. See the following
example, where for clarity symbols from $u$ are denoted by $x$:
\[
x\underbrace{ab\overbrace{abbccd}^{\text{depth $2$}}bccd}_{\text{depth
    $1$}}\overbrace{abbccd}^{\text{depth
    $1$}}xx\overbrace{abbbcccd}^{\text{depth $1$}}xx.
\]
Then it is possible to obtain $w$ from $u$ by performing all depth-$1$
insertions first, then all depth-$2$ insertions, and so on until $w$
is reached.

Suppose that, in order to obtain $w$ from $u$, a word $ab^\alpha
c^\alpha d \in G$ is inserted for some $\alpha > n$. Let $w =
w'aw''dw'''$, where these distinguished letters $a$ and $d$ are the
first and last letters of this inserted word. Notice that $w''
\derives_{\rel{R}} b^\alpha c^\alpha$, since
\[
w = w'aw''dw'' \derives_{\rel{R}} w'ab^\alpha c^\alpha dw''' \imderives_{\rel{R}} w'w''' \derives_{\rel{R}} u.
\]
(Of course, $w''$ may or may not contain inserted words of greater
depth.) Since $\alpha$ exceeds $n$, the automaton $\fsa{A}$ enters the
same state immediately after reading two different symbols $b$ of this
inserted word, say after reading $w'apb$ and
$w'apbqb$. Similarly it enters the same state immediately after
reading two different symbols $c$ of this inserted word, say after
reading $w'apbqbrc$ and $w'apbqbrcsc$. Therefore by the pumping lemma,
$w$ factors as $w'apbqbrcsctdw'''$ such that
\[
w'apb(qb)^irc(sc)^jtdw''' \in L
\]
for all $i,j \in \nset \cup \{0\}$, where the subwords $p$ and $q$
consist of letters $b$ (members of this inserted word) and possibly
also inserted words of greater depth, the subwords $s$ and $t$ consist
of letters $c$ (members of this inserted word) and possibly also
inserted words of greater depth, and the subword $r$ consists of some
letters $b$ followed by some letters $c$ (members of this inserted
word) and possibly also inserted words of greater depth. Thus
\[
p \derives_{\rel{R}} b^{\beta_1}, \qquad q \derives_{\rel{R}} b^{\beta_2},\qquad r \derives_{\rel{R}} b^{\beta_3} c^{\gamma_3}, \qquad s \derives_{\rel{R}} c^{\gamma_2}, \qquad t \derives_{\rel{R}} c^{\gamma_1},
\]
where $\beta_1 +\beta_2+\beta_3 + 2 = \gamma_1+\gamma_2+\gamma_3 + 2 =
\alpha$. It follows that
\begin{align*}
&\;\phantom{\derives_{\rel{R}}}\;w'apb(qb)^irc(sc)^jtdw'''\\
&\derives_{\rel{R}} w'ab^{\beta_1}b(b^{\beta_2}b)^ib^{\beta_3}c^{\gamma_3}c(c^{\gamma_2}c)^jc^{\gamma_1}dw''' \\
&\;\rlap{$=$}\phantom{\derives_{\rel{R}}}\; w'ab^{\alpha + (\beta_2+1)(i-1)}c^{\alpha + (\gamma_2+1)(j-1)}dw'''.
\end{align*}
Set $i = \gamma_2 + 2$ and $j = \beta_2 + 2$ to see that
\[
w'apb(qb)^{\gamma_2+1}rc(sc)^{\beta_2+1}tdw''' \in L
\]
and
\begin{align*}
&\;\phantom{\derives_{\rel{R}}}\;w'apb(qb)^{\gamma_2+1}rc(sc)^{\beta_2+1}tdw''' \\
&\derives_{\rel{R}} w'ab^{\alpha +
    (\beta_2+1)(\gamma_2+1)}c^{\alpha + (\gamma_2+1)(\beta_2+1)}dw''' \\
&\derives_{\rel{R}} w'w'' \qquad \text{(since $ab^{\alpha +
    (\beta_2+1)(\gamma_2+1)}c^{\alpha + (\gamma_2+1)(\beta_2+1)}d \in G$)} \\
&\derives_{\rel{R}} u.
\end{align*}
So there are two distinct words $w$ and
$w'apb(qb)^{\gamma_2+1}rc(sc)^{\beta_2+1}tdw''$ in $L$ representing the
same element $u$ of $M$. This is a contradiction and so shows the falsity
of the supposition that the insertion of a word $ab^\alpha c^\alpha d$
with $\alpha > n$ is used in obtaining the representative in $L$ from
a normal form word in $M$.

Let $G' = \{ab^\alpha c^\alpha d : \alpha \leq n\}$. Then obtaining a
word $w \in L$ representing $u \in M$ requires inserting only words
from $G' \subset G$.

Now suppose that an insertion of depth greater than $n^2$ is required
to obtain $w$ from $u$. Then $w$ factorizes as $w'apaqdrdw''$, where
the first distinguished letter $a$ and second distinguished letter $d$
are the first and last letters of some inserted word of depth $k$, and
the second distinguished letter $a$ and first distinguished letter $d$
are from some inserted word of depth $\ell > k$, and where the automaton
$\fsa{A}$ enters the same state after reading the two distinguished
letters $a$ and enters the same state after reading the two
distinguished letters $d$. (Such a factorization must exist because
there are only $n^2$ possible pairs of states, and there are inserted
words of depth exceeding $n^2$.) Notice that $aqd \derives_{\rel{R}} \emptyword$
and so $apaqdrd \derives_{\rel{R}} aprd \derives_{\rel{R}} \emptyword$. Then, by the
pumping lemma,
\[
w'apapaqdrdrdw'' \in L,
\]
but
\[
w'apapaqdrdrdw'' \derives_{\rel{R}} w'apaprdrdw'' \derives_{\rel{R}} w'aprdw'' \derives_{\rel{R}} w'w'' \derives_{\rel{R}} u, 
\]
and so there are two representatives $w$ and $w'apapaqdrdrdw''$ in $L$
of $u \in M$. This is a contradiction and so shows the falsity of the
assumption that insertions of depth greater than $n^2$ are required to
obtain the representative in $L$ of a normal form word in $M$.

Suppose that, in the process of performing insertions to obtain a
representative $w \in L$ for an element $u \in M$, a word $w^{(k)}$ is
obtained after the insertions of depth $k$ have been
performed. Suppose further that in performing the insertions of depth
$k+1$, more than $n$ insertions are made between consecutive letters
of $w^{(k)}$ to obtain a word $w^{(k+1)}$. (The reasoning below also
applies if $w^{(k)}$ is the empty word, which would require $k=0$.)
Then $w^{(k+1)}$ factors as
\[
w^{(k+1)} = v'ab^{\alpha_1}c^{\alpha_1}dab^{\alpha_2}c^{\alpha_2}d\cdots ab^{\alpha_h}c^{\alpha_h}dv'',
\]
where $h > n$, and each $ab^{\alpha_i}c^{\alpha_i}d$ is a word from $G'$. Then $w$ factors as
\[
w = w'ap_1dap_2d\cdots ap_hdw'',
\]
where $w' \derives_{\rel{R}} v'$, $w'' \derives_{\rel{R}} v''$, and $p_i \derives_{\rel{R}}
b^{\alpha_i}c^{\alpha_i}$ for each $i$.  Then $\fsa{A}$ enters the
same state on reading $w'ap_1dap_2d\cdots ap_id$ and
$w'ap_1dap_2d\cdots ap_jd$ for some $i < j$. So by the pumping lemma,
\[
q = w'ap_1dap_2d\cdots ap_id(ap_{i+1}d\cdots ap_jd)^2 ap_{j+1}d\cdots ap_kdw'' \in L.
\]
But
\begin{align*}
q&\;\rlap{$=$}\phantom{\derives_{\rel{R}}}\;w'ap_1dap_2d\cdots ap_id(ap_{i+1}d\cdots ap_jd)^2 ap_{j+1}d\cdots ap_kdw'' \\
&\derives_{\rel{R}} v'ab^{\alpha_1}c^{\alpha_1}dab^{\alpha_2}c^{\alpha_2}d\cdots \\
&\qquad\qquad \cdots ab^{\alpha_i}c^{\alpha_i}d(ab^{\alpha_{i+1}}c^{\alpha_{i+1}}d \cdots \\
&\qquad\qquad\qquad\qquad \cdots ab^{\alpha_j}c^{\alpha_j}d)^2 ab^{\alpha_{j+1}}c^{\alpha_{j+1}}d \cdots ab^{\alpha_h}c^{\alpha_h}dv'',\\
&\derives_{\rel{R}} v'v''\\
&\derives_{\rel{R}} u,
\end{align*}
and so there are two representatives $w$ and $q$ of the element $u \in
M$. This contradicts the uniqueness of representatives in $L$ and
shows the falsity of the supposition that more that $n$ insertions
between consecutive letters in the process of obtaining a
representative in $L$ for an element of $M$.

Therefore, to sum up: a representative $w$ in $L$ of an element $u$ of
$M$ can be obtained by inserting elements of $G'$ to a depth of at
most $n^2$, with at most $n$ consecutive words being inserted
between adjacent letters at any stage. Notice that the maximum length
of words in $G'$ is $2n+2$. Thus, starting with empty word, the after
depth~$1$ insertions, there are at most $n(2n+2)$ letters; after
depth~$2$ insertions, at most $n^2(2n+2)^2$; and after depth $n^2$
insertions, at most $h = n^{n^2}(2n+2)^{n^2}$. Similarly, if one starts
with a word $u$ and performs insertions to obtain its representative in
$L$, at most $h$ new symbols are inserted between
any adjacent pair of letters in $u$.

Define
\[
H = \big\{w \in A^* : |w| \leq h, w \derives_{\rel{R}} \emptyword\big\}.
\]
Then, by the observations in the last paragraph, if $u \in M$ with $u
= u_1\cdots u_n$ is represented by $w \in L$, then $w \in
Hu_1Hu_2\cdots Hu_nH$. Define the rational relation
\[
\rel{P} = \big(\{(a,a) : a \in A\} \cup \{(p,\emptyword) : p \in H\}\big)^*.
\]
Then, since removing all subwords in $H$ from a word in $L$ yields the
word to which it rewrites, it follows that
\[
M = (L \circ \rel{P}) \cap (A^* - A^*HA^*) = \bigl\{u \in A^* - A^*HA^* : (\exists w \in L)((w,u) \in \rel{P})\bigr\},
\]
and so $M$, which is the language of normal forms of $(A,\rel{R})$, is
regular.

However, two words $ab^\alpha c^\beta d$ and $ab^{\alpha'} c^{\beta'}
d$ (where $\alpha,\beta,\alpha',\beta' \in \nset$) represent the same
element of $M$ if and only if $\alpha = \beta$ and $\alpha' = \beta'$,
in which case they both represent the identity of $M$. Thus, since in
$M$ the unique representative of the identity is $\emptyword$, the
language $K = ab^*c^*d - M$, which is also regular,
consists of precisely those words of the form $ab^\alpha c^\beta d$
that represent the identity. That is, the language $K$ is $\{ab^\alpha
c^\alpha d : \alpha \in \nset\}$, which is not regular by the pumping
lemma. This is a contradiction, and so $M$ does not admit a regular
language of unique representatives.
\end{proof}

\bibliography{\jobname,automaticsemigroups,languages,presentations,semigroups,c_publications}
\bibliographystyle{alphaabbrv}

\end{document}